\documentclass[paper=a4, fontsize=11pt]{scrartcl} 

\usepackage[T1]{fontenc} 
\usepackage{fourier} 
\usepackage[english]{babel} 
\usepackage{amsmath,amsfonts,amsthm} 
\usepackage{amsmath, calligra, mathrsfs}
\usepackage{amssymb}
\usepackage[mathscr]{euscript}
\usepackage{verbatim}
\usepackage[dvipsnames]{xcolor}
\usepackage{mdframed} 
\usepackage{soulutf8}
\usepackage{stmaryrd}
\usepackage{tikz-cd}
\usepackage{hyperref}
\usepackage{lipsum} 

\usepackage{sectsty} 
\allsectionsfont{\centering \normalfont\scshape} 

\usepackage{fancyhdr} 
\usepackage{xurl}
\newcommand*{\sheafhom}{\mathcal{H}\kern -.5pt om}
\numberwithin{equation}{section} 
\numberwithin{figure}{section} 
\numberwithin{table}{section} 

\newtheorem{thm}{Theorem}[section]
\newtheorem{cor}[thm]{Corollary}
\newtheorem{prop}[thm]{Proposition}

\newtheorem{lem}[thm]{Lemma}

\theoremstyle{definition}
\newtheorem{defn}[thm]{Definition}

\newtheorem{exmp}[thm]{Example}

\theoremstyle{remark}
\newtheorem{rem}[thm]{Remark}

\DeclareMathOperator{\St}{st}

\DeclareMathOperator{\lk}{lk}

\DeclareMathOperator{\sgn}{sgn}

\DeclareMathOperator{\cha}{char}

\newcommand{\horrule}[1]{\rule{\linewidth}{#1}} 

\title{	
	\normalfont \normalsize 
	\textsc{} \\ [25pt] 
	\horrule{0.5pt} \\[0.4cm] 
	\huge Explicit expressions for the gamma vector leading to connections to upper/lower bounds and structural properties
	
	\horrule{2pt} \\[0.5cm] 
}

\author{Soohyun Park  \\ \href{mailto:soohyun.park@mail.huji.ac.il}{soohyun.park@mail.huji.ac.il}  } 

\date{\normalsize March 25, 2024} 

\begin{document}
	
	\maketitle 
	
	\begin{abstract}
		\noindent We find an explicit formula for the gamma vector in terms of the input polynomial in a way that extends it to arbitrary polynomials. More specifically, we find explicit linear combination in terms of coefficients of the input polynomial (using Catalan numbers and binomial coefficients) and an expression involving the derivative of the input polynomial. The first expression suggests connections to common Coxeter group/noncrossing partition structures in existing gamma positivity examples. In the case where the input is the $h$-polynomial of a simplicial complex, this gives an interpretation of the gamma vector as a measure of differences in local and global contributions. We also apply them to connect signs/inequalities of (shifts of) the gamma vector to upper/lower bound conditions on coefficients of the input polynomial. Finally, we make use of the shape of the sums used to make these estimates and connections with intersection numbers to relate these properties of the gamma vector to algebraic structures (e.g. characteristic classes involved in existing log concavity and Schur positivity properties).
	\end{abstract}

	\section*{Introduction}

	The main object that we are concerned with is the gamma vector associated to (reciprocal) polynomial $h(t)$ of degree $n$, which is the unique polynomial $\gamma$ of degree at most $ \lfloor \frac{\deg n}{2} \rfloor$ such that \[ h(t) = (1 + t)^ n \gamma \left( \frac{t}{(1 + t)^2} \right). \] This is how it was defined by Gal \cite{Gal} in the context of to real-rootedness questions of $h$-polynomials of flag spheres and had lower bound results with $h$-vectors in mind related to postivity of the gamma vector. In earlier work of Foata--Sch\"utzenberger and Strehl in connection to Eulerian polynomials, the gamma vector appears as the coefficients $\gamma_i$ used to write $h(t)$ as a linear combination of $t^i (1 + t)^{n - 2i}$ when $h(t)$ is a reciprocal polynomial.  Properties of this vector have since been involved in many different examples in combinatorics with the perspective taken by Gal tying together many objects studied earlier \cite{Athgam}. \\
	
	Since the gamma vector of a polynomial is defined in terms of a functional equation involving the starting polynomial, it is natural to ask if there are explicit expressions in terms of coefficients of the original polynomial. We find explicit formulas for the gamma vector in terms of the coefficients of $h(t)$ in a way that complements the connections to lower bounds and derivatives involving the input polynomial mentioned above. For example, the fact that we get negativity for large ratios (e.g. from considering face vectors as in Corollary \ref{recipfpol}) indicates that the symmetry of the coefficients or the underlying structure of a sphere play a role. In addition, they can be used to extend the definition of the gamma vector to arbitrary polynomials. Concretely, the components of the gamma vector/coefficients of $\gamma(u)$ written as a linear combination of the coefficients of $h(t)$ have coefficients coming from products of Catalan numbers and binomial coefficients (Theorem \ref{gammacat}). Note that the $n^{\text{th}}$ Catalan number enumerates the number of noncrossing partitions of $\{ 1, \ldots, n \}$ and that noncrossing partitions and Coxeter groups feature in many examples where gamma positivity is observed (e.g. Coxeter complexes and Eulerian polynomials, Coxeter--Narayana polynomials and the poset of noncrossing partitions \cite{Athgam}). Further details on relations to existing gamma positivity examples and other fields are discussed in Remark \ref{catcoxrem}.  \\
	
	Alternatively, we can describe the gamma vector in terms of the derivative of $\frac{h(t)}{(1 + t)^d}$ (Theorem \ref{gamderco}). Our applications mainly make use of this second formula. When $h(t)$ is the $h$-polynomial of a simplicial complex, we can use this to show that the gamma vector measures differences between local and global information stored by the $h$-vector (Proposition \ref{hderloc}). Afterwards, we use it to connect signs and inequalities involving the gamma vector to upper and lower bound properties of the coefficients of the input polynomial $h(t)$ (Proposition \ref{shiftgam}, Theorem \ref{boundgam}). The sums that we worked with while making the estimates were similar to those expressing $h$-vectors of a given simplicial complex in terms of its $f$-vector. We use this to show that the gamma vector of a polynomial has a simple expression in terms of the $h$-vector and $f$-vector of an auxiliary simplicial complex when the coefficients of the input polynomial are sufficiently large and we can restrict to looking at $h$-vectors when the coefficients of the input polynomial also satisfy certain divisibility properties (Proposition \ref{gamauxpo}). \\

	Continuing in this direction, we also consider interactions with algebraic structures. Since many of the examples we had considered earlier were connected to independence complexes and broken circuits of matroids (which can be express in terms of intersection numbers), we were motivated to consider an input polynomial with coefficients coming from intersections of divsiors on a variety. So, we considered the case where the input polynomial is the specialization $V(u) = (Au + B)^d$ of the volume polynomial associated to a pair of divsiors $A, B$ on a smooth projective variety. In general, the gamma vector has the same sign as $(-1)^r$ when the intersections $A^k B^{d - k}$ alternate in sign. Suppose that these intersections are nonnegative. When we are working with $\mathbb{P}^n$, we can relate the sign of the gamma vector to the relative ``size'' of $A$ compared to $B$ (since these correspond to integers) (Example \ref{gamvolmult}. When the ratio is very large, the gamma vector is negative. If it is very small, it is alternating in sign like the negative ratio case. Positivity properties of divisors are also related to the bounds in a more general setting since the $A^k B^{d - k}$ form a log concave sequence when $A$ and $B$ are nef. This can be used to reduce upper and lower bound conditions related to alternating sums of positive multiples of such terms to those on the first or last pair of elements in a sequence. Note that the connections to upper/lower bounds studied in the previous section carry over to volume polynomials with some additional interpretations and simplifications. Finally, we give an explanation of the $\mathbb{P}^n$ case in terms of Segre classes of vector bundles and make use of these classes to connect the rational functions giving components of the gamma vector in the derivative expression to positivity questions (e.g. log concavity) related to Segre classes and indicate a general connection between signs of the gamma vector to Schur positivity for nef divisors (Example \ref{gamseg}).

	\section*{Acknowledgments}
	
	I'd like to thank Lutz Warnke for pointing out a more recent reference \cite{SW} describing an elementary approach to the Lagrange inversion formula. Also, I'd like to thank Jon McCammond for mentioning connections between free probability and noncrossing partitions. This work was supported by Horizon Europe ERC Grant number: 101045750 / Project acronym: HodgeGeoComb. 
	
	\color{black}
	\section{The gamma vector as a linear combination of the $h$-vector components and local-global differences}
	
	\subsection{An explicit formula involving Catalan numbers, binomial coefficients, and derivatives}
	
	Given a reciprocal polynomial $h(t)$ of degree $n$, the $\gamma$-vector is defined (Proposition 2.1.1 on p. 272 of \cite{Gal}) as the coefficients of (unique) the polynomial $\gamma$ (of degree $\le \lfloor \frac{n}{2} \rfloor$) such that \[ h(t) = (1 + t)^n \gamma \left( \frac{t}{(1 + t)^2} \right). \] We find an explicit formula for the gamma vector as a linear combination of the $h_k$ with coefficients coming from Catalan numbers and binomial coefficients. Note that it gives an extension of the gamma vector to $h$-vectors that aren't necessarily symmetric.

	\begin{thm} \label{gammacat}
		Writing $n = \deg h$ for a reciprocal polynomial $h(t)$, we have that  
		
		\begin{align*}
			\gamma_m &= \sum_{i = 0}^m \sum_{\ell = 0}^i h_\ell \binom{-n}{i - \ell} \sum_{j_1 + \ldots + j_i = m} \widetilde{C}_{j_1} \cdots \widetilde{C}_{j_i} \\
			&= \sum_{i = 0}^m  \sum_{\ell = 0}^i h_\ell (-1)^{i - \ell} \binom{n + i - \ell - 1}{i - \ell} \sum_{j_1 + \ldots + j_i = m, j_r \ge 0} \widetilde{C}_{j_1} \cdots \widetilde{C}_{j_i},
		\end{align*}
		
		where $\widetilde{C}_k$ is the $k^{\text{th}}$ Catalan number if $k \ge 1$ and $\widetilde{C}_0 = 0$. Since $\deg \gamma \le \lfloor \frac{n}{2} \rfloor$ if $h$ is reciprocal, this sum is equal to $0$ for $m > \lfloor \frac{n}{2} \rfloor$ in that case. \\

	\end{thm}

	\begin{proof}
		We start with the identity \[ h(t) = (1 + t)^n \gamma \left( \frac{t}{(1 + t)^2} \right). \] Setting $u = \frac{t}{(1 + t)^2}$, we would like to ``solve'' for $t$ in terms of $u$. Solving the resulting quadratic equation gives \[ t = -1 + \frac{1 \pm  \sqrt{1 - 4u} }{2u}.  \] 
		
		Both sign choices will give $\frac{t}{(1 + t)^2} = u$. From now on, we will set \[ t = -1 + \frac{1 - \sqrt{1 - 4u}}{2u}. \]  Note that the second term is the generating function $C(u)$ for the Catalan numbers (e.g. see p. 4 of \cite{StCat}). Let $C_j$ be the $j^{\text{th}}$ Catalan number. \\
		
		To obtain the claimed relation, we rewrite the first identity as \[ \frac{h(t)}{(1 + t)^n} = \gamma \left( \frac{t}{(1 + t)^2} \right). \] After determining the coefficient of $t^k$ on the left hand side, we will rewrite everything in terms of $u$ by making the substitution $t = C(u) - 1$.\\
		
		If $k \le n$, the coefficient of $t^k$ in $\frac{h(t)}{(1 + t)^n}$ is
		
		\begin{align*}
			\sum_{i = 0}^k h_i \binom{-n}{k - i} &= \sum_{i = 0}^k h_i (-1)^{k - i} \binom{n + k - i - 1}{k - i}.
		\end{align*}
		
		If $k > n$, then the coefficient of $t^k$ in $\frac{h(t)}{(1 + t)^n}$ is
		
		\begin{align}
			\sum_{i = 0}^n h_i \binom{-n}{k - i} &= \sum_{i = 0}^n h_i (-1)^{k - i} \binom{n + k - i - 1}{k - i}
		\end{align}
		
		since $h(t)$ can only contribute terms of degree $\le n$. \\
		
		Let $A_k$ be the coefficient of $t^k$ in $\frac{h(t)}{(1 + t)^n}$. Making the substitution $t = C(u) - 1$ and noting that $C_0 = 1$, we have that the coefficient of $u^m$ on the left hand side from $\frac{h(t)}{(1 + t)^n}$ is
		
		\begin{align*}
			\gamma_m &= \sum_{i = 0}^m A_i \sum_{j_1 + \ldots + j_i = m, j_r \ge 0} \widetilde{C}_{j_1} \cdots \widetilde{C}_{j_i} \\
			&= \sum_{i = 0}^m  \sum_{\ell = 0}^{\min(n, i)} h_\ell \binom{-n}{i - \ell}   \sum_{j_1 + \ldots + j_i = m, j_r \ge 0} \widetilde{C}_{j_1} \cdots \widetilde{C}_{j_i} \\
			&= \sum_{i = 0}^m  \sum_{\ell = 0}^{\min(n, i)} h_\ell (-1)^{i - \ell} \binom{n + i - \ell - 1}{i - \ell} \sum_{j_1 + \ldots + j_i = m, j_r \ge 0} \widetilde{C}_{j_1} \cdots \widetilde{C}_{j_i}
		\end{align*}

	\end{proof}
	
	To write out the sum more explicitly, we use a convolution formula for Catalan numbers and the binomial theorem. 
	
	\begin{thm} \label{cataconv} (Theorem 2 on p. 1 of \cite{Reg}, \cite{GL}, Lemma 27 on p. 51 of \cite{BR}) \\
		Let $1 \le k \le n$. Then \[ \sum_{i_1 + \ldots + i_k = n, i_u \ge 0} C_{i_1 - 1} \cdots C_{i_k - 1} = \frac{k}{2n - k} \binom{2n - k}{n}. \] 
		
		Note that the coefficient of $u^r$ in $uC(u)$ is $C_{r - 1}$. This means that the sum computes the coefficient of $u^n$ in $(u C(u))^k = u^k C(u)^k$. Equivalently, it is the coefficient of $u^{n - k}$ in $C(u)^k$. Substituting in $r = n - k$. the coefficient of $u^r$ in $C(u)^k$ is \[  \frac{k}{k + 2r} \binom{k + 2r}{k + r}  \] since $2n - k = k + 2r$ and $n = k + r$. \\
		
		This coefficient can also be expressed as a multiple of a Catalan number depending on the parity of $k$ (Lemma 27 on p. 51 of \cite{BR}). More specifically, we have \[ \sum_{i_1 + \ldots + i_m = n, i_u \ge 0} C_{i_1} \cdots C_{i_m} = \begin{cases} 
			\frac{  m(n + 1)(n + 2) \cdots (n + \frac{m}{2} - 1)  }{ 2(n + \frac{m}{2} + 2)(n + \frac{m}{2} + 3) \cdots (n + m)  } C_{n + \frac{m}{2}} =  \frac{m}{2}  \frac{   \binom{  n + \frac{m}{2} - 1 }{  \frac{m}{2} - 1 }   }{  \binom{ n + m  }{  \frac{m}{2} - 1  }  }    C_{ n + \frac{m}{2} }  & \text{ if $m$ is even} \\
			\frac{  m(n + 1)(n + 2) \cdots (n + \frac{m - 1}{2}) }{ (n + \frac{m + 3}{2})( n + \frac{m + 3}{2} + 1  ) \cdots (n + m)   } C_{ n + \frac{m - 1}{2}  } = m  \frac{ \binom{n + \frac{m - 1}{2}}{ \frac{m - 1}{2}  }    }{  \binom{n + m}{   \frac{m - 1}{2}  }  }     C_{ n + \frac{m - 1}{2}  }  & \text{ if $m$ is odd} 
		\end{cases} \]
		
		A formula that covers both cases would be \[    \sum_{i_1 + \ldots + i_m = n, i_u \ge 0} C_{i_1} \cdots C_{i_m}  =   \frac{m}{  1 + \frac{1 - (-1)^{m + 1}}{2} }     \frac{    \binom{  n + \lfloor \frac{m}{2} \rfloor + \frac{(-1)^{m + 1} - 1}{2}     }{   \lfloor \frac{m}{2} \rfloor + \frac{(-1)^{m + 1} - 1}{2}     }       }{   \binom{n + m}{    \lfloor \frac{m}{2} \rfloor + \frac{(-1)^{m + 1} - 1}{2}    }    }      C_{ n + \lfloor \frac{m}{2} \rfloor  }    \]
		
		Some combinatorial interpretations of the sum are given in \cite{Ted}. Another perspective involving properties of posets is given in \cite{AW}. 
	\end{thm}

	Returning to the original sum, we would like to find the coefficient of $u^m$ in $(C(u) - 1)^i$. Using the binomial expansion, we have
	
	\begin{align*}
		(C(u) - 1)^i = \sum_{a = 0}^i \binom{i}{a} C(u)^a (-1)^{i - a}.
	\end{align*}
	
	
	Combining this with Theorem \ref{cataconv}, we can obtain the coefficient of $u^m$ in $(C(u) - 1)^i$ as a linear combination of the coefficients of $u^m$ in $C(u)^a$ for $0 \le a \le i$. This means that the coefficient of $u^m$ in $(C(u) - 1)^i$ is \[ \sum_{a = 0}^i \binom{i}{a} (-1)^{i - a} \frac{a}{a + 2m} \binom{a + 2m}{a + m} \]
	
	after setting $k = a$ and $n = m$ in Theorem \ref{cataconv}. Substituting this into Proposition \ref{gammacat}, we have the following:
	
	\begin{cor}
		In the context of Proposition \ref{gammacat}, we have 
		
		\begin{align*}
			\gamma_m &= \sum_{i = 0}^m \sum_{\ell = 0}^{\min(n, i)} h_\ell \binom{-n}{i - \ell} \sum_{a = 0}^i \binom{i}{a} (-1)^{i - a} \frac{a}{a + 2m} \binom{a + 2m}{a + m} \\
			&= \sum_{i = 0}^m  \sum_{\ell = 0}^{\min(n, i)} h_\ell (-1)^{i - \ell} \binom{n + i - \ell - 1}{i - \ell} \sum_{a = 0}^i \binom{i}{a} (-1)^{i - a} \frac{a}{a + 2m} \binom{a + 2m}{a + m}.
		\end{align*}
		
	\end{cor}
	

	In particular, the coefficient of $h_\ell$ in $\gamma_m$ is 
	
	\begin{align*}
		\sum_{i = \ell}^m \binom{-n}{i - \ell} \sum_{j_1 + \ldots + j_i = m} \widetilde{C}_{j_1} \cdots \widetilde{C}_{j_i} &= \sum_{i = \ell}^m  (-1)^{i - \ell} \binom{n + i - \ell - 1}{i - \ell}  \sum_{j_1 + \ldots + j_i = m} \widetilde{C}_{j_1} \cdots \widetilde{C}_{j_i} \\
		&= \sum_{i = \ell}^m \binom{-n}{i - \ell} \sum_{a = 0}^i \binom{i}{a} (-1)^{i - a} \frac{a}{a + 2m} \binom{a + 2m}{a + m} \\
		&= \sum_{i = \ell}^m  (-1)^{i - \ell} \binom{n + i - \ell - 1}{i - \ell} \sum_{a = 0}^i \binom{i}{a} (-1)^{i - a} \frac{a}{a + 2m} \binom{a + 2m}{a + m}.
	\end{align*}
	
	Since the power series $\widetilde{C}(u) = C(u) - 1$ satisfies the functional equation $\widetilde{C}(u) = u(\widetilde{C}(u) + 1)^2$, we can make a further simplification of the sum $\sum_{j_1 + \ldots + j_i = m} \widetilde{C}_{j_1} \cdots \widetilde{C}_{j_i} $ (the coefficient of $u^m$ in $\widetilde{C}(u)^i$) using Lagrange inversion:
	
	\begin{thm} \label{laginvthm} (Lagrange inversion formula, Theorem 5.4.2 on p. 38 of \cite{StEC2} and Theorem 2.1.1 on p. 214 of \cite{Ges}, Theorem 1 on p. 944 of \cite{SW}) \\
		Let $k$ be a field with $\cha k = 0$. Suppose $G(x) \in k[[x]]$ with $G(0) \ne 0$, and let $f(x)$ be defined by \[ f(x) = x G(f(x)). \] Then \[  n[x^n] f(x)^k = k[x^{n - k}] G(x)^n. \] Here, $[x^a] Q(x)$ denotes the coefficient of $x^a$ in $Q(x)$. 
	\end{thm}
	
	This can be used to make simplifications to the sum defining the gamma vector.
	
	\begin{cor}
		
		We have that 
		
		\begin{align*}
			\gamma_m &= \sum_{i = 0}^m \sum_{\ell = 0}^i h_\ell \binom{-n}{i - \ell}  \frac{i}{m} \binom{2i}{m - i} \\
			&= \sum_{i = 0}^m \sum_{\ell = 0}^i h_\ell (-1)^{i - \ell} \binom{n + i - \ell - 1}{i - \ell} \frac{i}{m} \binom{2i}{m - i}. 
		\end{align*}
		
		Note that $\min(n, i) = i$ since $0 \le i \le m$ and $m \le \lfloor \frac{n}{2} \rfloor$, which means that $i \le \lfloor \frac{n}{2} \rfloor < n$. \\
		
		This means that the coefficient of $h_\ell$ in $\gamma_m$ is \[ \sum_{i = \ell}^m \binom{-n}{i - \ell}  \frac{i}{m} \binom{2i}{m - i} = \sum_{i = \ell}^m (-1)^{i - \ell} \binom{n + i - \ell - 1}{i - \ell} \frac{i}{m} \binom{2i}{m - i}  \]
	\end{cor}
	
	\begin{proof}
		In our case, we can set $G(w) = (1 + w)^2$ and  \[  \sum_{j_1 + \ldots + j_i = m} \widetilde{C}_{j_1} \cdots \widetilde{C}_{j_i}  = [u^m]  \widetilde{C}(u)^i  = \frac{i}{m} \binom{2i}{m - i}.  \]
		
		Substituting this into previous computations gives the stated sums. \\
	\end{proof}
	
	
	In addition to a simplification of the expression for the coefficients $\gamma_m$, a corollary of Lagrange inversion can be used to give a ``functional'' description of where these sums come from. \\
	
	\begin{cor} \label{laginvcor} (Corollary 5.4.3 on p. 42 of \cite{StEC2}, Theorem 2.1.1 on p. 214 of \cite{Ges}, Theorem 1 on p. 994 of \cite{SW}) \\
		Suppose that we are in the setting of Theorem \ref{laginvthm}. Then for any power series $H(x) \in k[[x]]$ (or Laurent series $H(x) \in k((x))$), we have \[  n [x^n] H(f(x)) = [x^{n - 1}] H'(x) G(x)^n,  \] where $f(x) = x G(f(x))$. 
	\end{cor}
	
	Here is our application to the $\gamma$-vector.
	
	\begin{thm} \label{gamderco}
		Let $J(v) = \frac{h(v)}{(1 + v)^n}$. Then, we have \[ r \gamma_r = [u^{r - 1}] J'(u) (u + 1)^{2r}.  \] 
	\end{thm}
	
	\begin{proof}
		Recall that we have $G(w) = (w + 1)^2$ and $f(u) = \widetilde{C}(u) = C(u) - 1$, where $C(u)$ is the generating function for the Catalan numbers. Note that \[ \frac{h(t)}{(1 + t)^n} = \gamma \left( \frac{t}{(1 + t)^2} \right)  \] since $h(t) = (1 + t)^n \gamma \left( \frac{t}{(1 + t)^2}  \right)$. As mentioned previously, setting $t = \widetilde{C}(u) = C(u) - 1$ means that $u = \frac{t}{(1 + t)^2}$. \\
		
		If we set \[ J(v) = \frac{h(v)}{(1 + v)^n},  \] then $\gamma(u) = J(\widetilde{C}(u))$ and Corollary \ref{laginvcor} implies the statement. \\
	\end{proof}

	\begin{rem} \textbf{(Connections to noncrossing partitions)} \\ \label{catcoxrem}
		The Catalan numbers hint at some connections to existing gamma positivity examples. Recall that the $n^{\text{th}}$ Catalan number enumerates the number of noncrossing partitions of $[n] = \{ 1, \ldots, n \}$. Noncrossing partitions and Coxeter groups feature in many examples connected to gamma positivity (e.g. Eulerian polynomials and $h$-vectors of type A Coxeter complexes, Coxeter--Narayana polynomials and the poset of noncrossing partitions in \cite{Stem} and Section 2.2 on p. 24 -- 27 of \cite{Athgam}). Also, it would be interesting if there is a relation to higher moments in the context of connections between free probability and noncrossing partitions \cite{Ni}, which includes applications (\cite{Bia}, \cite{NS}, \cite{BGGJS}) of the Lagrange inversion formula used to obtain the derivative formula for $r \gamma_r$ in Theorem \ref{gamderco}. Finally, noncrossing partition statistics also feature in an expression for a version of the toric $h$-vector for cubical complexes \cite{BiH}, which are part of the geometric motivation for the study of flag complexes (see discussion on locally CAT(0) spaces on p. 270 of \cite{Gal}). 
	\end{rem}

	\subsection{Local-global differences}
	
	Let's take a closer look at the derivative of the function $J(v) = \frac{h(v)}{(1 + v)}$ in the expression $J'(u)(u + 1)^{2r}$ (coefficient of $u^{r - 1}$ is $r\gamma_r$). The quotient rule implies that \[ J'(v) = \frac{(1 + v)h'(v) - n h(v)}{(1 + v)^{n + 1}}. \] Note that $r \le \lfloor \frac{n}{2} \rfloor$. This means that $2r \le n < n + 1$ and $J'(u) (u + 1)^{2r}$ has a positive power of $u + 1$ in the denominator and $(1 + u)h'(u) - n h(u)$ in the numerator. In the numerator, the coefficient of $v^k$ is $(k + 1)h_{k + 1} - (n - k) h_k$.  Before starting more explicit computations, we make a heuristic comment on the interpreation of the derivative. In the case of $f$-polynomials, the derivative of the $f$-polynomial gives the sum of the $f$-polynomials of the links over the vertices of the given simplicial complex. This is a special case of the result below.

	\begin{defn} (Definition 5.3.4 on p. 232 of \cite{BH}, p. 60 of \cite{St}) \\
		Let $\Delta$ be a simplicial complex and $F$ be a a subset of the vertex set of $\Delta$. The \textbf{star of $\mathbf{F}$} is the set $\St_\Delta(F) = \{ G \in \Delta : F \cup G \in \Delta \}$, and the \textbf{link of $\mathbf{F}$} is the set $\lk_\Delta(F) = \{ G : F \cup G \in \Delta, F \cap G = \emptyset \}$. In other words, we are removing $F$ from elements of $\St_\Delta(F)$.  
	\end{defn}

	\begin{prop} \label{derivloc} (Gal, Lemma 2.2.1 on p. 274 of \cite{Gal}) \\
		The sum of the $f$-polynomials of the links of all faces with $k$vertices of any simplicial complex $X$ is equal to $\frac{f_X^{(k)}(t)}{k!}$.
	\end{prop}
	
	In fact, combining this result with the usual identity relating the reciprocal polynomials of the $f$-polynomial and $h$-polynomial gives an interpretation of the numerator of $J'(v)$ related to a difference between local and global properties (e.g. related to the partition complex) for the case of $h$-vectors. \\
	
	\begin{lem} \label{fderhloc}
		In the setting above, we have \[ (v - 1)^d f' \left( \frac{1}{v - 1} \right)  = (v - 1) \sum_{p \in V(\Delta)} v^{d - 1} h_{\lk_\Delta(p)} \left( \frac{1}{v} \right) \] and \[  (v - 1)^{d - 1} f' \left( \frac{1}{v - 1} \right)  =  \sum_{p \in V(\Delta)} v^{d - 1} h_{\lk_\Delta(p)} \left( \frac{1}{v} \right)  \]
	\end{lem}
	
	\begin{proof}
		Recall that $(v - 1)^d f \left( \frac{1}{v - 1} \right) = v^d h \left( \frac{1}{v}  \right)$. Substituting in $\frac{1}{v - 1}$ in place of $v$ in Proposition \ref{derivloc}, we have
		
		\begin{align*}
			f' \left( \frac{1}{v - 1} \right) &= \sum_{p \in V(\Delta)} f_{\lk_\Delta(p)} \left( \frac{1}{v - 1} \right) \\
			\Longrightarrow (v - 1)^d f' \left( \frac{1}{v - 1} \right) &= (v - 1)^d   \sum_{p \in V(\Delta)} f_{\lk_\Delta(p)} \left( \frac{1}{v - 1} \right) \\
			&= (v - 1) \sum_{p \in V(\Delta)} (v - 1)^{d - 1}  f_{\lk_\Delta(p)} \left( \frac{1}{v - 1} \right) \\
			&= (v - 1) \sum_{p \in V(\Delta)} v^{d - 1} h_{\lk_\Delta(p)} \left( \frac{1}{v} \right) 
		\end{align*}
	\end{proof}
	
	In order to look at the numerator of $J'(v)$, we obtain a counterpart of Proposition \ref{fderhloc} for the $h$-polynomial.
	
	\begin{prop} \label{hderloc}
		Given a simplicial complex $\Delta$, we have that \[ h'(w) = \frac{1}{1 - w} \left( \sum_{p \in V(\Delta)} h_{\lk_\Delta(p)}(w) - d h(w)  \right). \]
	\end{prop}
	
	\begin{proof}
		We take the derivative of both sides of the identity $(v - 1)^d f \left( \frac{1}{v - 1} \right) = v^d h \left( \frac{1}{v}  \right)$. Then, we have 
		
		\begin{align*}
			d(v - 1)^{d - 1} f \left( \frac{1}{v - 1} \right) + (v - 1)^d f' \left( \frac{1}{v - 1} \right) \cdot -\frac{1}{(v - 1)^2} &= d v^{d - 1} h \left( \frac{1}{v} \right) + v^d h' \left( \frac{1}{v} \right) \cdot - \frac{1}{v^2} \\
			\Longrightarrow d(v - 1)^{d - 1} f \left( \frac{1}{v - 1} \right) - (v - 1)^{d - 2} f' \left( \frac{1}{v - 1} \right) &= d v^{d - 1} h \left( \frac{1}{v} \right) - v^{d - 2} h' \left( \frac{1}{v} \right) \\
			\Longrightarrow (v - 1)^{d - 1} \left( d f \left( \frac{1}{v - 1} \right) - \frac{1}{v - 1} f' \left( \frac{1}{v - 1} \right)   \right) &= v^{d - 1} \left( d  h \left( \frac{1}{v} \right)  - \frac{1}{v} h' \left( \frac{1}{v} \right)  \right). 
		\end{align*}
		
		Substituting in $w = \frac{1}{v}$, this can be rewritten as
		\begin{align*}
			\left(  \frac{1}{w} - 1  \right)^{d - 1}  \left(  d f \left(  \frac{1}{  \frac{1}{w} - 1   }  \right) -  \frac{1}{  \frac{1}{w} - 1   } f' \left(     \frac{1}{  \frac{1}{w} - 1   }   \right)   \right) &= \frac{1}{w^{d - 1}} (d h(w) - w h'(w)) \\
			\Longrightarrow \left( \frac{1 - w}{w}  \right)^{d - 1} \left(  d f \left(  \frac{w}{1 - w}  \right) - \frac{w}{1 - w} f' \left( \frac{w}{1 - w}  \right)   \right) &= \frac{1}{w^{d - 1}} (d h(w) - w h'(w)) \\
			\Longrightarrow d h(w) - w h'(w) &= (1 - w)^{d - 1}  \left(  d f \left(  \frac{w}{1 - w}  \right) - \frac{w}{1 - w} f' \left( \frac{w}{1 - w}  \right)   \right) \\
			\Longrightarrow  w h'(w)  - d h(w) &= (1 - w)^{d - 1} \left(     \frac{w}{1 - w} f' \left( \frac{w}{1 - w}  \right)  -  d f \left(  \frac{w}{1 - w}  \right)     \right) \\
			&= \frac{w}{1 - w}  \cdot (1 - w)^{d - 1}  f' \left( \frac{w}{1 - w}  \right)  - d (1 - w)^{d - 1} f \left(  \frac{w}{1 - w}  \right) \\
			&= \frac{w}{1 - w} \sum_{p \in V(\Delta)} h_{\lk_\Delta(p)}(w) - d (1 - w)^{d - 1} f \left(  \frac{w}{1 - w}  \right) \\
			&= \frac{w}{1 - w} \sum_{p \in V(\Delta)} h_{\lk_\Delta(p)}(w) - \frac{d}{1 - w} h(w) \\
			\Longrightarrow w h'(w) &= \frac{w}{1 - w} \sum_{p \in V(\Delta)} h_{\lk_\Delta(p)}(w) - \frac{d}{1 - w} h(w) + d h(w) \\
			&= \frac{w}{1 - w} \sum_{p \in V(\Delta)} h_{\lk_\Delta(p)}(w) + d h(w) \left( 1 - \frac{1}{1 - w} \right) \\
			&= \frac{w}{1 - w} \sum_{p \in V(\Delta)} h_{\lk_\Delta(p)}(w) - \frac{dw}{1 - w} h(w) \\
			&= \frac{w}{1 - w} \left(  \sum_{p \in V(\Delta)} h_{\lk_\Delta(p)}(w)  - d h(w)  \right) \\
			\Longrightarrow h'(w) &= \frac{1}{1 - w} \left(  \sum_{p \in V(\Delta)} h_{\lk_\Delta(p)}(w)  - d h(w)  \right),  
		\end{align*}
		
		where the first line with the $h$-polynomials of links over vertices follows from Lemma \ref{fderhloc} with $\frac{1 - w}{w}$ substituted in place of $v - 1$ and the line after it uses the reciprocal relation $(v - 1)^d f \left( \frac{1}{v - 1} \right) = v^d h \left( \frac{1}{v}  \right)$ with $v - 1$ replaced by $\frac{1 - w}{w}$. 
	\end{proof}

	\begin{exmp} \label{locglobgam} \textbf{(Gamma vectors measuring the difference between local and global contributions to the $h$-polynomial)} \\
		Going back to the numerator of \[ J'(v) = \frac{(1 + v)h'(v) - n h(v)}{(1 + v)^{n + 1}}, \] Proposition \ref{hderloc} implies that \[  (1 + w)h'(w) - d h(w) = \frac{1 + w}{1 - w} \sum_{p \in V(\Delta)} h_{\lk_\Delta(p)}(w) - \frac{2d}{1 - w} h(w).  \]
		
		Comparing this to Proposition \ref{fderhloc}, the derivative of the $h$-polynomial measures the \emph{difference} between the ``global'' $h$-vector and those of its links over its vertices (in comparison to the derivative of the $f$-polynomial which takes the \emph{sum} of the contribution of the links over vertices). Combining this with Theorem \ref{gamderco}, the same can be said about components $\gamma_r$ of the gamma vector. \\
	\end{exmp}
	
	\section{ Inequalities and auxiliary $h$-vector representations for gamma vectors of polynomials }

	We use the formula from the previous section to study properties of the gamma vector while keeping examples from independence and broken circuit complexes of matroids in mind. We give explicit bounds using translations and upper/lower bounds and express the gamma vector in terms of $h$-vectors of auxiliary simplicial posets.
	
	\subsection{ Bounds involving translations }
	
	We first start with some explicit computations related to the (extended) gamma vector and independence/broken circuit complexes of matroids. We will focus on $f$-vectors of these complexes, which are specializations of the Tutte polynomial. A similar analysis seems to hold for $f$-vectors of general shellable simplicial complexes (modulo bounds on the last two coordinates). \\
	
	We will show the following:
	
	\begin{prop} \label{shiftgam}
		Suppose that $A(t) = a_0 + a_1 t + \ldots + a_d t^d$ is a polynomial with nonnegative coefficients $a_i \ge 0$. Let $R(t) = t^d A(t^{-1})$ be the reciprocal polynomial of $A$ and let $B(t) := R(t + 1)$. \\
		
		Then, $\gamma_{\frac{d}{2}}(B) \le 0$ and $\gamma_r(B) \le 0$ for odd $r$. We also have $\gamma_r(B) \le 0$ for even $r$ if \[ \frac{a_{k + 2r - 1}}{a_k} \ge  \frac{ k \binom{2r - k - 1}{r - 1} }{(k + 2r - 1) \binom{k + r - 1}{r - 1} } \] for all $0 \le k \le \min(2r - 1, d - 2r + 1)$. \\
	\end{prop}

	\begin{cor} \label{recipfpol} 
		If $r$ is odd or $r = \frac{d}{2}$, $\gamma_r(F) \le 0$ for the reciprocal polynomial $F$ of the $f$-polynomial of any Cohen--Macaulay poset $P$ (e.g. shellable simplicial complexes). For example, this applies to $F(t) = \chi_M(-t)$, where $\chi_M(t)$ is the characteristic polynomial of a loopless matroid $M$ after substituting in its broken circuit complex. Also, there are infinitely Cohen--Macaulay posets where $\gamma_r(F) \le 0$ for all $0 \le r \le \frac{d}{2}$. \\
	\end{cor}

	\begin{proof}
		The first part is a combination of Proposition \ref{shiftgam} and well-known result on broken circuit complexes and characteristic polynomials (Proposition 3.1 on p. 424 of \cite{Bry}). As for the second part, we use the fact that any vector $h = (h_0, h_1, \ldots, h_d) \in \mathbb{Z}^{d + 1}$ with $h_0 = 1$ and $h_i \ge 0$ for all $i$ is the $h$-vector of a Cohen--Macaulay simplicial poset (Theorem 3.1 on p. 326 of \cite{Stfh}). 
	\end{proof}
	
	\begin{rem}
		Recall that the gamma vector has coefficients coming from derivatives involving the ``input'' polynomial (Theorem \ref{gamderco}). Both of examples above are from specializations of the Tutte polynomial of a matroid. Note that partial derivatives of Tutte polynomials give generating functions related to internal/external activities \cite{LV}. \\
	\end{rem}
	
	\begin{proof} (Proof of Proposition \ref{shiftgam}) \\
		By Theorem \ref{gamderco}, we have that \[ r \gamma_r = [u^{r - 1}] J'(u) (1 + u)^{2r}, \] where \[ J(u) = \frac{ u^d B(u^{-1})}{(1 + u)^d}. \] 
		
		Since $B(t) = R(t + 1)$, we have that \[ \frac{B(u)}{(1 + u)} = \frac{C(u + 1)}{(u + 1)^d} = \widetilde{J}(u + 1), \] where \[ \widetilde{J}(w) = w^{-d} R(w) = \sum_{i = 0}^d c_{-i} w^i.  \]

		Finding the coefficient of $u^{r - 1}$ above for the gamma vector (substituting $w = 1 + u$) is the same as taking the coefficient of $(w - 1)^{r - 1}$ in the Taylor expansion of \[ \left( \frac{d}{dw} a_{-i} w^i \right) \cdot w^{2r} \] around 1. \\
		
		Before applying the translation, we have

		\begin{align*}
			\widetilde{J}'(w) &= \sum_{i = 0}^d -i a_i w^{-i - 1} \\
			&= \sum_{j = - d - 1}^{-1} (j + 1) a_{- j - 1} w^j \\
			\Longrightarrow w^{2r} \widetilde{J}'(w) &= \sum_{j = - d - 1}^{-1} (j + 1) a_{-j - 1} w^{j + 2r} \\
			&= \sum_{\ell = -d - 1 + 2r}^{-1 + 2r} (\ell - 2r + 1) a_{- \ell + 2r - 1} w^\ell,
		\end{align*}
		
		where we made the substitutions $j = - i - 1$ and $\ell = j + 2r$. \\
		
		To find the coefficient of $u^{r - 1}$ in the original function before the translation, we take the Taylor expansion of $w^{2r} \widetilde{J}'(w)$ around 1. This means taking the $(r - 1)^{\text{th}}$ and dividing by $(r - 1)!$ before substituting in $w = 1$ to find the coefficient of $(w - 1)^{r - 1}$ in the expansion. Since \[ \frac{1}{(r - 1)!} \frac{d}{dw^{r - 1}} w^k = \binom{k}{r - 1} w^{k - r + 1}, \] this means that 
		
		\begin{align*}
			\frac{1}{(r - 1)!} \frac{d}{dw^{r - 1}} w^{2r} \widetilde{J}'(w) &= \sum_{\ell = -d - 1 + 2r}^{-1 + 2r} (\ell - 2r + 1) \binom{\ell}{r - 1}  a_{- \ell + 2r - 1} w^{\ell - r + 1} \\
			&= \sum_{p = -d + r}^r (p - r) \binom{p + r - 1}{r - 1} a_{-p + r} w^a,
		\end{align*}
		
		where we made the substitution $p = \ell - r + 1$. \\
		
		To find $r \gamma_r$ associated to $B(t) := R(t + 1)$ for the reciprocal polynomial $R(t) := t^d A(t^{-1})$ of $A$, we compute the coefficient of $(w - 1)^{r - 1}$ in the Taylor expansion of $w^{2r} \widetilde{J}'(w)$ around 1 by setting $w = 1$:
		
		\begin{align*}
			r \gamma_r &= \sum_{p = -d + r}^r (p - r) \binom{p + r - 1}{r - 1} a_{-p + r} \\
			&= \sum_{k = 0}^d -k \binom{2r - k - 1 }{r - 1} a_k,
		\end{align*}
		
		where we made the substitution $k = -p + r$. \\
		
		Under the assumptions of our problem, $a_k \ge 0$ for all $0 \le k \le d$. If $k \le 2r - 1$, the binomial coefficients are nonnegative. For example, this is the case when $r = \frac{d}{2}$. If $r$ is odd, then all the binomial coefficients are nonnegative since negative binomial coefficients from $k > 2r - 1$ would be of the form $(-1)^{r - 1} \binom{-2r + k + 1 + r - 1}{r - 1} = (-1)^{r - 1} \binom{k - r}{r - 1} \ge 0$ and $(-1)^{r - 1} = 1$ if $r$ is odd.  This means that $\gamma_r \le 0$ if $r = \frac{d}{2}$ or $r$ is odd. \\
		
		Suppose that $r$ is even. Then, we need to compare the sizes of the positive and negative terms. Take $0 \le k \le 2r - 1$. The index $k$ term is \[ k \binom{2r  -k  - 1}{r - 1} a_k \] and the index $k + 2r - 1$ term is \[ (k + 2r - 1) \binom{-k}{r - 1} a_{k + 2r - 1} = (k + 2r - 1) (-1)^{r - 1} \binom{k + r - 2}{r - 1} a_{k + 2r - 1} . \] In order for the second term to lie in the indices between $0$ and $d$, we also need $k \le d - 2r + 1$. Since $r$ is even, the second term would be negative. So, it suffices to have
		
		\begin{align*}
			k \binom{2r - k - 1}{r - 1} a_k &\le (k + 2r - 1) \binom{k + r - 2}{r - 1} a_{k +2r - 1} \\
			\Longleftrightarrow \frac{a_{k + 2r - 1}}{a_k} &\ge \frac{ k \binom{2r - k - 1}{r - 1} }{(k + 2r - 1) \binom{k + r - 1}{r - 1} }
		\end{align*}
				
		in order to have $\gamma_r \le 0$ for even $r$ as well. This is a sort of lower bound condition.
		
	\end{proof}

	\subsection{  Upper/lower bounds and general gamma vectors }

	Given a polynomial $B(u)$, recall that \[ r \gamma_r = [u^{r - 1}] J'(u) (1 + u)^{2r} = [u^{r - 1}] \frac{(1 + u) B'(u) - dB(u)}{(1 + u)^{d - 2r + 1}} \]by Proposition \ref{gammacat}. \\
	
	To study the sign of the gamma vector of $B$, we will show that $r \gamma_r$ takes the following form:
	
	\begin{prop} \label{ftypesum}
		Given a polynomial $B(u) = b_0 + b_1 u + \ldots + b_d u^d$ of degree $d$, we have that  
		
		\begin{align*}
			r \gamma_r(B) &= (-1)^r \sum_{k = 1}^{r - 1} (-1)^k k \binom{d - r - k - 1}{r - k} b_k + (-1)^r d \sum_{k = 0}^{r - 1} (-1)^k \binom{d - r - k - 1}{r - k - 1} b_k \\
			\Longrightarrow r \gamma_r - (-1)^r d \binom{d - r - 1}{r - 1} b_0  &= (-1)^r \sum_{k = 1}^{r - 1} (-1)^k k \binom{d - r - k - 1}{r - k} b_k + (-1)^r d \sum_{k = 1}^{r - 1} (-1)^k \binom{d - r - k - 1}{r - k - 1} b_k. \\
		\end{align*}
		
	\end{prop}
	
	\begin{proof}
		Let $Q(u) = (1 + u) B'(u) - dB(u)$ be the numerator of the function above. The coefficient of $u^k$ in $Q(u)$ is $q_k = (k + 1) b_{k + 1} - (d - k) b_k$. To look at the gamma vector, we consider negative binomial coefficients to find that
		
		\begin{align*}
			r \gamma_r &= [u^{r - 1}] \frac{Q(u)}{(1 + u)^{d - 2r + 1}} \\
			&= \sum_{i + j = r - 1, i, j \ge 0} q_i \binom{  -(d - 2r + 1) }{j} \\
			&= \sum_{i + j = r - 1, i, j \ge 0} q_i (-1)^j \binom{d - 2r + 1 + j - 1}{j} \\
			&= \sum_{i = 0}^{r - 1} q_i (-1)^{r - 1 - i} \binom{d - 2r + r - 1 - i}{r - 1 - i} \\
			&= (-1)^{r - 1} \sum_{i = 0}^{r - 1} (-1)^i \binom{d - r - 1 - i}{r - 1 - i} q_i . 
		\end{align*}
		
		Substituting in $q_i = (i + 1) b_{i + 1} - (d - i) b_i$, we have
		
		\begin{align*}
			r \gamma_r &= (-1)^{r - 1} \sum_{i = 0}^{r - 1} (-1)^i \binom{d - r - 1 - i}{r - 1 - i} ( (i + 1) b_{i + 1} - (d - i) b_i).
		\end{align*}
		
		The coefficient of $b_0$ inside the sum is $-d \binom{d - r - 1}{r - 1}$. If $k \ge 1$, the coefficient of $b_k$ inside is the sum of the index $i = k - 1$ and $i = k$ terms of the sum given by
		
		\begin{align*}
			(-1)^{k - 1} k \binom{d - r - 1 - k + 1}{r - 1 - k + 1}  - (-1)^k (d - k) \binom{d - r - 1 - k}{r - 1 - k} \\
			= (-1)^{k - 1} \left(  k \binom{d - r - k}{r - k}  + (d - k) \binom{d -r - 1 - k}{r - 1 - k} \right) \\
			= (-1)^{k - 1} \left(  k \left( \binom{d - r - k}{r - k}  - \binom{d - r - k - 1}{r - k - 1}  \right) + d \binom{d - r - k - 1}{r - k - 1} \right) \\
			= (-1)^{k - 1} \left(  k \binom{d - r - k - 1}{r - k} + d \binom{d - r - k - 1}{r - k - 1}  \right) 
		\end{align*}
		
		where we used the identity $\binom{p}{q} - \binom{p - 1}{q - 1} = \binom{p - 1}{q}$ in the last line. 
		
		This means that 
		\begin{align*}
			r \gamma_r &= (-1)^{r - 1} \sum_{k = 1}^{r - 1} (-1)^{k - 1} k \binom{d - r - k - 1}{r - k} b_k + (-1)^{r - 1} d \sum_{k = 0}^{r - 1} (-1)^{k - 1} \binom{d - r - k - 1}{r - k - 1} b_k \\
			&= (-1)^r \sum_{k = 1}^{r - 1} (-1)^k k \binom{d - r - k - 1}{r - k} b_k + (-1)^r d \sum_{k = 0}^{r - 1} (-1)^k \binom{d - r - k - 1}{r - k - 1} b_k \\
			&=  (-1)^r \sum_{k = 1}^{r - 1} (-1)^k k \binom{d - r - k - 1}{r - k} b_k + (-1)^r d \sum_{k = 1}^{r - 1} (-1)^k \binom{d - r - k - 1}{r - k - 1} b_k \\ 
			&+ (-1)^r d \binom{d - r - 1}{r - 1} b_0 \\
			\Longrightarrow r \gamma_r - (-1)^r d \binom{d - r - 1}{r - 1} b_0  &= (-1)^r \sum_{k = 1}^{r - 1} (-1)^k k \binom{d - r - k - 1}{r - k} b_k + (-1)^r d \sum_{k = 1}^{r - 1} (-1)^k \binom{d - r - k - 1}{r - k - 1} b_k .
		\end{align*}
	\end{proof}
	
	Given this simplification, we will study the sign of $\gamma_r$ from a couple of different points of view. We start by looking at ratios of consecutive terms as in the previous section.
	
	\begin{lem} \label{ratiosign}
		Let $\{a_k\}_{k = 0}^N$ be a finite sequence of nonnegative real numbers $a_k \ge 0$. Given a number $M$, write $\sgn(M)$ for the sign of $M$.
		
		\begin{enumerate}
			\item (Increasing sequence) If $a_k \le a_{k + 1}$ for $0 \le k \le N$, then \[ \sgn \left( \sum_{k = 0}^N (-1)^k a_k \right) = (-1)^N.  \]
			
			\item (Decreasing sequence) If $a_k \ge a_{k + 1}$ for $0 \le k \le N$, then \[ \sum_{k = 0}^N (-1)^k a_k \ge 0. \]
		\end{enumerate}
		
	\end{lem}

	\begin{rem} \label{altshift} 
		If we had started with $k = 1$ instead in Part 2, we would switch the sign in Part 2. In general, the sign depends on whether we started with a negative or positive term of the alternating sum. However, there is no change in the statement of Part 1 even after we change the starting index for the alternating sum. \\
	\end{rem}
	
	This implies the following:
	
	\begin{thm} \label{boundgam}
		Let $B(t) = b_0 + b_1 t + \ldots + b_d t^d$ be a polynomial of degree $d$.
		
		\begin{enumerate}
			\item If $\sgn(b_k) = (-1)^k$, then $\sgn(\gamma_r) = (-1)^r$ for all $0 \le r \le \frac{d}{2}$. 
			
			\item Suppose that $b_k \ge 0$ for all $0 \le k \le d$.
			
				\begin{enumerate}
					\item If $\{ b_k \}$ is a decreasing sequence for $1 \le k \le r - 1$, then \[ \sgn  \left(   r \gamma_r - (-1)^r d \binom{d - r - 1}{r - 1} b_0  \right)  = (-1)^{r + 1} \] if $0 \le r \le \frac{d}{3}$ or $r = \frac{d}{2}$. For $r =  \frac{d}{2}$, we have that $\sgn(\gamma_{  \frac{d}{2}  }) = (-1)^r$. This also holds for $\frac{d}{3} < r < \frac{d}{2}$ if \[ \frac{b_k}{b_{k + 1}} \ge \frac{(k + 1)(r - k)}{k(d - r - k - 1)} \] for all $0 \le k \le r - 1$. \\
					
					\item If $\{ b_k \}$ is an increasing sequence for $1 \le k \le r - 1$and \[ \frac{b_{k + 1}}{b_k} \ge \frac{d - r - k - 1}{r - k} \] for all $0 \le k \le r - 1$, then $\gamma_r \le 0$. If $r = \frac{d}{2}$, we have that $\gamma_{  \frac{d}{2}  } \le 0$. \\

				\end{enumerate}
			
		\end{enumerate}
	\end{thm}
	
	\begin{rem}
		Some examples where Part 1 would apply is for include $\pm \overline{\chi}_M(t)$ and a translate $\pm \overline{\chi}_M(t + 1)$ for reduced characteristic polynomials of matroids $M$ (with sign depending on the parity of the rank of $M$). As for Part 2b), we can think of face vectors where the entries increase quickly (Theorem \ref{fhex}). Examples for individual indices coming from independence complexes of matroids are in equality cases of upper bounds on face vectors of independence complexes of matroids in work of Castillo--Samper (Theorem 3.3 on p. 10 of \cite{CAS}). We can move in between cases 2a) and 2b) by looking at reciprocal polynomials. \\
	\end{rem}
	
	\begin{proof}
		Given the simplification above, it is clear that Part 1 holds. Also, Lemma \ref{ratiosign} directly implies the statements in Part 2 for $r = \frac{d}{2}$ since $(-1)^r r \gamma_r$ is a positive multiple of the alternating sum of the $b_i$ in this case.  For the remaining indices in Part 2, we start by noting that the binomial terms in each sum of Proposition \ref{ftypesum} give a decreasing sequence as the index increases since \[ \frac{ \binom{a + 1}{b + 1} }{ \binom{a}{b} } = \frac{a + 1}{b + 1} \ge 1 \] if $a \ge b \ge 0$ and $r \le \frac{d}{2}$. In Part 2a), Lemma \ref{ratiosign} implies that it is enough to say that the second sum in Proposition \ref{ftypesum} has sign $(-1)^r$. As for the first sum, we need to show that the increase given by multiplying by $\frac{k + 1}{k}$ still leaves the ratio $\le 1$. Having the ratio of the index $k$ term to the index $k + 1$ term being $\ge 1$ in the first sum of Proposition \ref{ftypesum} is equivalent to 
		
		\begin{align*}
			\frac{ k \binom{d - r - k - 1}{r - k} }{ (k + 1) \binom{d - r -  k - 2}{r - k - 1} } &=  \frac{k}{k + 1} \cdot \frac{d - r - k - 1}{r - k} \\
			&\ge 1 \\
			\Longleftrightarrow k(d - r - k - 1) &\ge (k + 1)(r - k)  \\
			dk - rk - k^2 - k &\ge rk - k^2 + r - k \\
			\Longleftrightarrow dk - rk &\ge rk + r  \\
			\Longleftrightarrow dk - 2rk &\ge r \\
			\Longleftrightarrow k(d - 2r) &\ge r \\
			\Longleftrightarrow k &\ge \frac{r}{d - 2r}. 
		\end{align*}
		
		In other words, the ``good'' indices of the first sum in Proposition \ref{ftypesum} where the coefficient of $b_k$ is larger than that of $b_{k + 1}$ are given by those where $k \ge \frac{r}{d - 2r}$. If $r \le \frac{d}{3}$, then we have that $3r \le d$ and $r \le d - 2r$. This means that $\frac{r}{d - 2r} \le 1$ and the lower bound $k \ge \frac{r}{d - 2r}$ is trivial since the sum is taken over $k \ge 1$. In these cases, the sequence of elements of the first sum in Proposition \ref{ftypesum} is decreasing and Lemma \ref{ratiosign} implies that $\sgn(\gamma_r(B)) = (-1)^r$. \\
		
		Finally, we consider Part 2b). Note that the coefficients of the $b_k$ are decreasing in the second sum in the first line of Proposition \ref{ftypesum}. In order for the terms to be increasing, we need to have \[ \frac{b_{k + 1}}{b_k} \ge \frac{ \binom{d - r - k - 1}{r - k - 1} }{ \binom{d - r - k - 2}{r - k - 2} } = \frac{d - r - k - 1}{r - k - 1}. \] If $0 \le r \le \frac{d}{3}$, the coefficients of the first sum in the first line of Proposition \ref{ftypesum} are also decreasing. So, the condition for the terms of the first sum to increase is to have \[ \frac{b_{k + 1}}{b_k} \ge \frac{k}{k + 1} \cdot \frac{d - r - k - 1}{r - k - 1}. \] However, it turns out that this condition is redundant since \[ \frac{k}{k + 1} \cdot \frac{d - r - k - 1}{r - k} \le \frac{d - r - k - 1}{r - k - 1}. \] This follows directly from clearing denominators since $k(r - k - 1) = kr - k^2 - k \le (k + 1)(r - k) = kr - k^2 + r - k$. The proof of Part 2a) implies that we have shown the statement of Part 2b) for $0 \le r \le \frac{d}{3}$. To look at the remaining indices, we can split the first sum in Proposition \ref{ftypesum} into those where $k \ge \frac{r}{d - 2r}$ (decreasing coefficients of $b_k$) and $k < \frac{r}{d - 2r}$ (increasing coefficients of $b_k$). The terms that we take alternating sums of are increasing on each part. Lemma \ref{ratiosign} then shows that the terms have sign $(-1)^{r - 1}$ and we have the claimed inequality.
		
	\end{proof}

	\subsection{Gamma vectors and (modified) $h$-vector representations} 
	
	Recall from Proposition \ref{ftypesum} that \[ r \gamma_r(B) = (-1)^r \sum_{k = 1}^{r - 1} (-1)^k k \binom{d - r - k - 1}{r - k} b_k + (-1)^r d \sum_{k = 0}^{r - 1} (-1)^k \binom{d - r - k - 1}{r - k - 1} b_k \] for a polynomial $B(u) = b_0 + b_1 u + \ldots + b_d u^d$. \\
	
	This was obtained from \[ r \gamma_r = (-1)^{r - 1} \sum_{i = 0}^{r - 1} (-1)^i \binom{d - r - 1 - i}{r - 1 - i} q_i  \] with $q_i = q_i = (i + 1) b_{i + 1} - (d - i) b_i$. \\
	
	In our analysis of the gamma vector from this perspective, we have mainly considered sums of the form \[ \sum_{i = 0}^{r - 1} (-1)^i \binom{P - i}{Q - i} M_i  \] for some $M_i \ge 0$ with one of the following collections of parameters:
	
	\begin{itemize}
			\item Expression from Proposition \ref{ftypesum}:
			
				\begin{itemize}
					\item $P = d - r - 1$, $Q = r$, $M_0 = 0$, and $M_k = k b_k$ with $b_k \ge 0$  for $1 \le k \le r - 1$ from the first sum
					
					\item $P = d - r - 1$, $Q = r - 1$, and $M_k = b_k$ with $b_k \ge 0$ for $0 \le k \le r - 1$ from the second sum
				\end{itemize}
			
			\item Older formula: $P = d - r - 1$, $Q = r - 1$, and $q_i = (i + 1) b_{i + 1} - (d - i) b_i$
	\end{itemize}

	This has a shape similar to the formula \[ h_j(\Delta) = \sum_{i = 0}^j (-1)^{j - i} \binom{d - i}{j - i} f_{i - 1}(\Delta) \] expressing the $h$-vector of a simplicial complex $\Delta$ in terms of its $f$-vector (Lemma 5.1.8 on p. 213 of \cite{BH}). We can rewrite this as \[ (-1)^j h_j (\Delta)= \sum_{i = 0}^j (-1)^i \binom{d - i}{j - i} f_{i - 1}(\Delta) \] in a way that is more similar to the sums written above. \\
	
	If $M_i = f_{i - 1}(\Delta)$ for some simplicial complex $\Delta$ of dimension $P - 1$ and $Q = r - 1$, then \[ \sum_{i = 0}^{r - 1} (-1)^i \binom{P - i}{Q - i} M_i  = (-1)^{r - 1} h_{r - 1}(\Delta). \] Note that everything is well-behaved under restrictions to $m$-skeleta. In other words, we can keep adding suitable $M_k$ for higher indices $k$ while keeping the previous ones. \\
	
	We will focus on the expression from Proposition \ref{ftypesum} since want to look at sums of positive terms and the older formula may involve some negative terms. For the second sum in Proposition \ref{ftypesum}, it suffices to have $b_k = f_{k - 1}(\Delta)$ for some simplicial complex $\Delta$. It remains to consider the first sum. The main issue is that have a sum $1 \le k \le r - 1$ (meaning that 1 or 2 terms are missing) and $Q = r$ instead of $r - 1$. To ``complete'' this to something coming from a simplicial complex, we can either write $\binom{d -r - k - 1}{r - k } = \binom{d - r - k - 1}{r - k - 1} \cdot \frac{d - 2r - 1}{r - k}$ and add one additional term for $k = 0$ or keep the given binomial coefficient $\binom{d -r - k - 1}{r - k}$ and add 2 additional terms for $k = 0$ and $k = r$. We can combine these observations with the following result:
	
	\begin{thm} (Stanley, Theorem 2.1 on p. 321 of \cite{Stfh}) \label{fhex} \\
		Let $f = (f_0, f_1, \ldots, f_{d - 1}) \in \mathbb{Z}^d$. The following two conditions are equivalent:
		
		\begin{enumerate}
			\item There exists a simplicial poset $P$ of dimension $d - 1$ with $f$-vector $f(P) = f$. 
			
			\item $f_i \ge \binom{d}{i + 1}$ for $0 \le i \le d - 1$. 
		\end{enumerate}
		
	\end{thm}

	In this setting, we can express $r \gamma_r$ in terms of the degree $r - 1$ or $r$ part of the $h$-vector or $f$-vector of two auxiliary simplicial posets.
	
	\begin{prop} \label{gamauxpo}
		Let $B(u) = b_0 + b_1 u + \ldots + b_d u^d$ be a polynomial with nonnegative coefficients $b_i \ge 0$. 
		
		\begin{enumerate}
			\item If the $b_i$ are sufficiently large, there are auxiliary $(d - r - 2)$-dimensional simplicial posets $P$ and $Q$ such that \[ r \gamma_r = h_r(P) - f_{r - 1}(P) - d h_{r - 1}(Q) - (-1)^r \binom{d - r - 1}{r}. \] 
			
			\item If the $b_i$ are sufficiently large and each $b_i$ is a multiple of $r - i$, there are auxiliary $(d - r - 2)$-dimensional simplicial posets $R$ and $S$ such that \[ r \gamma_r = - h_{r - 1}(R) - d h_{r - 1}(S) - (-1)^r \binom{d - r - 1}{r - 1} .  \]
		\end{enumerate}
	\end{prop}

	\begin{proof}
		\begin{enumerate}
			\item Taking the missing $k = 0$ and $k = r$ terms into account, Theorem \ref{fhex} implies that there are $(d - r - 2)$-dimensional simplicial posets $P$ and $Q$ such that 
			
			\begin{align*}
				r \gamma_r(B) &= (-1)^r \left( (-1)^r h_r(P) - \binom{d - r - 1}{r} - (-1)^r f_{r - 1}(P) + (-1)^{r - 1} d h_{r - 1}(Q) \right) \\
				&= h_r(P) - f_{r - 1}(P) - d h_{r - 1}(Q) - (-1)^r \binom{d - r - 1}{r}.
			\end{align*}
		
			\item Similarly, we have $(d - r - 2)$-dimensional simplicial posets $R$ and $S$ such that 
			
			\begin{align*}
				r \gamma_r &= (-1)^r \left(  (-1)^{r - 1} h_{r - 1}(R) - \binom{d - r - 1}{r - 1} + (-1)^{r - 1} d h_{r - 1}(S) \right) \\
				&= - h_{r - 1}(R) - d h_{r - 1}(S) - (-1)^r \binom{d - r - 1}{r - 1} 
 			\end{align*}
		
		\end{enumerate}
	\end{proof}
	
	\color{black}
	
	\section{ Intersection numbers and volume polynomials }
	Following up to considering the shape and face vector interpretation of the sums used to study bounds on the gamma vector, we will look into sources of related algebraic structure. For example, some of our recurring examples involved independence complexes and broken circuit complexes of matroids (e.g. from characteristic polynomials). A special property of the $f$-vectors of independence complexes and $h$-vectors of broken circuit complexes is that they can be considered as intersection numbers by work of Adiprasito--Huh--Katz \cite{AHK} and Ardila--Denham--Huh (\cite{ADH1}, \cite{ADH2}). More specifically, we have that $f_i(IN(M)) = \alpha^{d - i} \beta^i$ in the Chow ring of the matroid $M$ and $h_i(BC(M)) = \gamma^i \delta^{n - i - 1}$ in the conormal Chow ring of the matroid.  \\
	
	In this context, it is natural to ask how the gamma vector behaves for other polynomials involving intersections of divisors (i.e. with coefficients from top degree intersections). For a specialization of the volume polynomial, the gamma vector plays a similar role to the $g$-vector of the broken circuit complex. We also consider examples related to how the sign of the gamma vector is related to positivity of divisors on projective space and give a geometric interpretation in terms of Segre classes of vector bundles. \\

	More specifically, we will consider a specialization of the volume polynomial.
	
	\begin{defn}
		Given divisors $A, B$ on a (smooth projective) variety $X$, let $V(u) = (uA + B)^d$ be the \textbf{specialized volume polynomial}.
	\end{defn}

	 The polynomial $V$ will take the role of $h(u)$ used to define the gamma vector. Recall  from Theorem \ref{gamderco} that \[ r \gamma_r = [u^{r - 1}] J'(u) (1 + u)^{2r}, \] where \[ J(u) = \frac{V(u)}{(1 + u)^d}. \] 
	
	Then, we have the following:
	
	\begin{prop} \label{volgamnum}
		Let $V(u) = (uA + B)^d$ for divisors $A$ and $B$ on a smooth projective variety $X$. Then, we have that \[ r \gamma_r = [u^{r - 1}] \frac{Q(u)}{(1 + u)^{d - 2r + 1}}, \] where $Q(u) = \sum_{k = 0}^d q_k u^k$ with \[ q_k = d \binom{d - 1}{k} (A^{k + 1} B^{d - k - 1} - A^k B^{d - k}). \] Also, $r \gamma_r = h(P)$ for some auxiliary $(d - r - 2)$-dimensional simplicial poset $P$ when $d$ is sufficiently large and the $A^k B^{d - k}$ are increasing for $0 \le k \le r - 1$.   
	\end{prop}

	\begin{rem}
		The second part was motivated by $f$-vectors of independence complexes of matroids (use shellability and Proposition 3 on p. 3931 of \cite{Cha}, \cite{AHK}) and $h$-vectors of broken circuit complexes of matroids and the analogue of this setup for divisors of the conormal Chow ring of a matroid (see \cite{JKL}, \cite{CAS}, \cite{ADH1}, \cite{ADH2}).. Some related further discussion is in Part 2 of Remark \ref{posmat}.
	\end{rem}
	
	\color{black} 
	\begin{proof}
		Let $Q(u) = (1 + u)V'(u) - d V(u)$ be the numerator of $J'(u)$ with denominator $(1 + u)^{d - 2r + 1}$ and $q_k$ be the coefficient of $u^k$ in $Q(u)$. Since $V'(u) = d A (tA + B)^{d - 1}$, we have that 
		
		\begin{align*}
			q_k &= [u^k] V'(u) + [u^{k - 1}] V'(u) + d [u^k] V(u) \\
			&= d \binom{d - 1}{k} A \cdot A^k B^{d - k - 1} + d \binom{d - 1}{k - 1} A \cdot A^{k - 1} B^{d - k} - d \binom{d}{k} A^k B^{d - k} \\
			&= d \binom{d - 1}{k} A^{k + 1} B^{d - k - 1} + d \binom{d - 1}{k - 1} A^k B^{d - k} - d \binom{d}{k} A^k B^{d - k} \\
			&= d \left( \binom{d - 1}{k} A^{k + 1} B^{d - k - 1} + \binom{d - 1}{k - 1} A^k B^{d - k} - \binom{d}{k} A^k B^{d - k} \right) \\
			&= d \left( \binom{d - 1}{k} A^{k + 1} B^{d - k - 1} - \binom{d - 1}{k} A^k B^{d - k} \right) \\
			&= d \binom{d - 1}{k} (A^{k + 1} B^{d - k - 1} - A^k B^{d - k}).
		\end{align*}
	
		The second statement is an application of Theorem \ref{fhex} as in Proposition \ref{gamauxpo}.
	\end{proof}

	We now consider a family of examples where the signs of the components gamma vector of the volume polynomial has a direct connection with postivity properties of the divisors involved. 
	
	
	\begin{exmp} \textbf{(Gamma vector signs and multiples)} \label{gamvolmult} \\
		
		To get an idea about quantitative properties related to signs of gamma vectors of volume polynomials, we consider the special case where $A^{k + 1} B^{d - k - 1}/A^k B^{d - k}$ doesn't depend on $k$. This includes cases where $A$ is a constant multiple $\rho$ of $B$ such as divisors $A, B$ on $\mathbb{P}^n$. In this setting, Proposition \ref{volgamnum} implies that \[ \frac{q_{k + 1}}{q_k} = \frac{d - k - 1}{k + 1} \rho \] and that our starting point/constant is $q_0 = d(AB^{d - 1} - B^d)$ and $q_k = \rho^k q_0 = d(AB^{d - 1} - B^d)$ for $0 \le k \le r - 1$. Since $0 \le k \le r - 1$ and $0 \le r \le \frac{d}{2}$, we have that $\frac{d - k - 1}{k + 1} \ge 1$. Recall that \[  r \gamma_r =  (-1)^{r - 1} \sum_{i = 0}^{r - 1} (-1)^i \binom{d - r-  i - 1}{r - i - 1} q_i. \] This means that the $(k + 1)^{\text{th}}$ term divided by the $k^{\text{th}}$ term (assuming latter nonzero) is \[ \frac{r - k - 1}{d - r- k - 1} \cdot \frac{d - k - 1}{k + 1} \rho. \] The binomial coefficients $\binom{d - r - i - 1}{r - i - 1}$ are decreasing as we increase $i$. Lemma \ref{ratiosign} then implies the following:
		
		\begin{enumerate}
			\item If $\rho < 0$ or $0 < \rho \ll 1$, then $\sgn(\gamma_r) = (-1)^{r - 1} \sgn(q_0) = (-1)^{r - 1} \sgn(AB^{d - 1} - B^d)$. This is from directly evaluating terms or using Lemma \ref{ratiosign}.
			
			\item If $\rho = 0$, then $\sgn(\gamma_r) = \sgn(q_0) = (-1)^{r - 1} \sgn(AB^{d - 1} - B^d)$.
			
			\item If $\rho \gg 0$, then $\sgn(\gamma_r) = \sgn(q_0) = \sgn(AB^{d - 1} - B^d)$. More specifically, it suffices to have $\rho$ such that \[ \frac{d - k - 1}{k + 1} \rho > \frac{d - r- k - 1}{r - k - 1} \] for $1 \le k \le r - 2$. Since $\frac{d - k - 1}{k + 1} \ge 1$, we can take $\rho > d - 2r + 1$ for a particular $r$ and $\rho > d$ for all these conditions to hold at once.

		\end{enumerate}
		
		In this setting, the sign of the gamma vector components is a measure of positivity properties of the common ratio term.
		
	\end{exmp}

	\begin{rem} \label{posmat} ~\\
		\vspace{-3mm}
		\begin{enumerate}
			\item There are some related sums involving a mixture of terms of the form $\rho^i$ and $f$-vectors/$h$-vectors in work of Chari (Proposition 3 on p. 3931 of \cite{Cha}) and Brown--Colbourne (Thoerem 3.1 on p. 575 of \cite{BC}). Also, more general positivity properties are connected to alternating sums of multiples of positive intersections of the form $A^k B^{n - k}$ (e.g.  in the context of gamma vectors). If $A$ and $B$ are nef, the sequence $a_k = A^k B^{d - k}$ forms a log concave sequence. Recall that a sequence $\{ a_k \}_{k = 1}^N$ is log concave if $a_k^2 \ge a_{k - 1} a_{k - 1}$ for all $2 \le k \le N - 1$. Since this is equivalent to $\frac{a_k}a_{k - 1} \ge \frac{a_{ k + 1 }}{a_k}$, the upper bound condition $\frac{a_{i + 1}}{a_i} \le C$ holds for all $1 \le i \le N - 1$ if and only if $\frac{a_1}{a_2} \le C$. Similarly, the lower bound condition $\frac{a_{i + 1}}{a_i} \ge C$ holds for $1 \le i \le N - 1$ if and only if $\frac{a_N}{a_{N - 1}} \ge C$. The is a different perspective for the case of nef divisors leading to connections to Schur positivity mentioned at the end of Example \ref{gamseg}. \\

			\item An analogue of this applies to $f$-vectors of independence complexes and $h$-vectors of broken circuit complexes of matroids by work of Adiprasito--Huh--Katz \cite{AHK} and Ardila--Denham--Huh (\cite{ADH1}, \cite{ADH2}).  In addition, we can make some observations on the meaning of the terms in Proposition \ref{volgamnum} in this setting. The entries of the $h$-vector of the broken circuit complex of the matroid are increasing from indices between 1 and $\frac{d}{2}$ (\cite{JKL}, \cite{CAS}). Also, the differences $A^{k + 1} B^{d - k - 1} - A^k B^{d - k}$ form $O$-sequence and give the $h$-vector of some Cohen--Macaulay simplicial complex in the case of independence complexes of matroids (Theorem 4.3 on p. 374 of \cite{Swg}, \cite{JKL}). \\
		\end{enumerate}

	\end{rem}
	
	The gamma vector associated to a volume polynomial specialization satisfies similar properties to those of $f$-vectors of independence complexes, $h$-vectors of broken circuit complexes, and characteristic polynomials of matroids studied in the previous section. \\
	
	\begin{cor} \label{volbd} ~\\
		Let $P_k = A^{k + 1} B^{d - k - 1} - A^k B^{d - k}$. The gamma vector associated the volume polynomial specialization satisfies the following properties: \\
		
		\begin{enumerate} 
			
			\item If $A^k B^{d - k}$ is a negative multiple of $A^{k + 1} B^{d - k - 1}$ for each $k$ (e.g. if $A$ is a negative multiple of $B$), the sign of $\gamma_r$ is $(-1)^{r - 1}$.
			
			\item If \[ \frac{d - k - 1}{k + 1} \frac{q_{k + 1}}{q_k} > \frac{d - r- k - 1}{r - k - 1}, \] then $\sgn(\gamma_r) = \sgn(q_0)$. 
			
			\item If the $A^k B^{d - k}$ are nonnegative and decreasing as $k$ increases, then $\sgn(\gamma_r) = (-1)^r$. 
			
		\end{enumerate}
	\end{cor}

	\color{black} 
	\begin{proof}
		In the notation above, we have 
		
		\begin{align*}
			r \gamma_r = [u^{r - 1}] \frac{Q(u)}{(1 + u)^{d - 2r + 1}} \\
			= \sum_{i + j = r - 1, i, j \ge 0} q_i \binom{-(d - 2r + 1)}{j} \\
			= (-1)^{r - 1} \sum_{i + j = r - 1, i, j \ge 0} (-1)^i q_i \binom{d - 2r + j}{j}  \\
			= (-1)^{r - 1} \sum_{i = 0}^{r - 1} (-1)^i \binom{d - r-  i - 1}{r - i - 1} q_i \\
			= (-1)^{r - 1} \sum_{i = 0}^{r - 1} (-1)^i \binom{d - r-  i - 1}{r - i - 1} q_i  \\
			= (-1)^{r - 1} d \sum_{k = 0}^{r -1} (-1)^k \binom{d - r - 1 - k}{r - 1 - k} \binom{d - 1}{k} (A^{k + 1} B^{d - k - 1} - A^k B^{d - k}) \\
			= (-1)^{r - 1} d \sum_{k = 0}^{r - 1} \binom{d - r - 1 - k}{r - 1 - k} \left(  \binom{d - 1}{k - 1} - \binom{d - 1}{k} \right) A^k B^{d - k} \\
			= (-1)^r  d \sum_{k = 0}^{r - 1} (-1)^k \binom{d - r - 1 - k}{r - 1 - k} \left( \binom{d - 1}{k} - \binom{d - 1}{k - 1} \right) A^k B^{d - k}
		\end{align*}
		
		for $0 \le r \le \frac{d}{2}$. Note that $\binom{d - r-  i - 1}{r - i - 1}$ decreases as $i$ increases since \[ \frac{\binom{a + 1}{b + 1}}{\binom{a}{b}} = \frac{a + 1}{b + 1} \ge 1 \] when $a, b \ge 0$. If $\sgn(q_k) = (-1)^k$ (e.g. when $A$ is a negative multiple of $B$), then $\sgn(\gamma_r) = (-1)^{r - 1}$ since each term of the sum is nonnegative. \\
		
		For $r = \frac{d}{2}$, we have
		
		\begin{align*}
			\frac{d}{2} \gamma_{\frac{d}{2}} &= (-1)^{r - 1} \sum_{ i = 0}^{r - 1} (-1)^i q_i.
		\end{align*}
		
	\end{proof}

	Using the volume polynomial as an input gives a way to interpret the sign of the gamma vector in terms of Segre classes associated to direct sums of line bundles or embeddings of linear subspaces of affine space in the setting of Example \ref{gamvolmult}. We also give some connections to questions to combinatorial positivity questions.  
	
	\begin{exmp} \textbf{(Gamma vectors for volume polynomials and characteristic classes)} \label{gamseg} \\
		As for a possible geometric interpretation of $\frac{Q(u)}{(1 + u)^{d - 2r - 1}}$ from \[ r \gamma_r = [u^{r - 1}] \frac{Q(u)}{(1 + u)^{d - 2r + 1}}, \]  note that $A^{k + 1} B^{d - k - 1} - A^k B^{d - k} = A^k B^{d - k - 1}(A - B)$. In the setting of Example \ref{gamvolmult} (e.g. for divisors on projective space), this means that we can formally take $\rho ``='' A B^{-1}$ and set $A^{k + 1} B^{d - k - 1} - A^k B^{d - k} = B^{d - 1}(A - B) \rho^k$. Since $q_k = d \binom{d - 1}{k} (A^{k + 1} B^{d - k - 1} - A^k B^{d - k})$  this would formally give an expression of the form \[ r \gamma_r \text{``='' } [u^{r - 1}] \frac{d B^{d - 1}(A - B)}{(1 + u)^{d - 2r + 1} (1 - \rho u)^{d - 1}}. \]
		
		The expression inside for $\frac{Q(u)}{(1 + u)^{d - 2r + 1}}$ has a similar shape to the Segre class of a direct sum of line bundles. A ``motivic'' Chern class associated to the inclusion of a linear subspace into affine space has a term that is similar to the denominator with $-t$ in place of $u$ and a parameter $y$ replacing $\rho$ (Example 5.6 on p. 244 of \cite{DM}). \\
		
		If $-\rho$ is treated like the first Chern class of a globally generated line bundle, we can think of $r \gamma_r$ as the degree $r - 1$ part $s_{r - 1}(\mathcal{E})$ of the Segre class of a direct sum $\mathcal{E}$ of globally generated line bundles (e.g. normal bundle of a subvariety given by the complete intersection of the corresponding divisors). Given a globally generated vector bundle $\mathcal{E}$, note that $(-1)^i s_i(\mathcal{E})$ is represented by a positive linear combination of components of a codimension $i$ subvariety where $r + i - 1$ general sections of $\mathcal{E}$ fail to generate $\mathcal{E}$ (p. 363 -- 364 of \cite{EH}). This is sort of analogous to the signs $(-1)^{r - 1}$ we obtained for the gamma vector components $\gamma_r$ when consecutive terms of the form $A^k B^{d - k}$ and $A^{k + 1} B^{d - k - 1}$ have opposite signs (e.g .when $A$ is a negative multiple of $B$). Finally, we note that connections between Segre classes and Lorentzian polynomials have been studied recently by Aluffi \cite{A}. More specifically, we have that terms of the form $\frac{Q(u)}{(1 + u)^{d - 2r + 1}}$ coming from Segre zeta functions have differences of terms that are log concave. There are always gamma vectors such that this property is satisfied. Given a polynomial $Q$ of degree $\deg Q \le d - 1$, there is always some $B$ such that $\frac{B(u)}{( 1 + u)^{d + 1}} du = \int \frac{Q(u)}{(1 + u)^{d + 1}} du$. Apart from this, the perspective of taking characteristic classes (without the degree restrictions implicit above) leads to connections with Schur positivity when the divisors involved are nef by work of Fulton--Lazarsfeld \cite{FL} and its generalization to nef divisors by Demailly--Peternell--Schneider \cite{DPS}. 
 	\end{exmp}

\end{document}